\documentclass[11pt,a4paper]{article}
\usepackage[utf8]{inputenc}
\usepackage{amsmath}
\usepackage{amscd}
\usepackage{amsthm}
\usepackage{amssymb}
\usepackage{tikz-cd}
\usepackage{amsfonts}
\usepackage[all]{xy}
\usepackage{mathtools}
\usepackage{multirow}
\usepackage{lipsum}
\usepackage{hyperref}
\usepackage{newtxtext,newtxmath}
\usepackage{graphicx}
\usepackage[]{units}
\usepackage{enumitem}
\usepackage{booktabs}
\usepackage{url}
\usepackage{times}
\newtheorem{theorem}{Theorem}[section]
\newtheorem{lemma}{Lemma}[section]

\newtheorem{remark}{Remark}
\newtheorem*{concluding remark}{Concluding Remark}
\newtheorem{theoremA}{Theorem}

\title{Higher symmetric power $L$-functions and their Fourier coefficients}
\author{Kampamolla Venkatasubbareddy\\Email: \href{mailto:20mmpp02@uohyd.ac.in}{20mmpp02@uohyd.ac.in} \and Ayyadurai Sankaranarayanan\\Email: \href{mailto:sank@uohyd.ac.in}{sank@uohyd.ac.in}}

\date{}

\begin{document}
\maketitle

\begin{abstract}
   Let $H_k$ be the set of all normalized primitive holomorphic cusp forms of even integral weight $k\geq 2$ for the full modular group $SL(2, \mathbb{Z})$, and let $j\geq 3$ be any fixed integer. For $f\in H_k$, we write $\lambda_{{\rm{sym}^j  }f}(n)$ for the $n^\textit{th}$ normalized Fourier coefficient of $L(s,{\rm{sym}}^j f)$. In this article, we establish an asymptotic formula for the sum 
   \begin{equation*}
       \sum_{\substack{n=a_1^2+a_2^2+\hdots+a_6^2\leq x\\\left(a_1,a_2,\hdots, a_6\right)\in \mathbb{Z}^6}}\lambda_{{\rm{sym}}^j  f}^2(n),
   \end{equation*}
   with an improved error term.
\end{abstract}

\footnote{2020 AMS \emph{Mathematics subject classification.} Primary 11F11, 11F30, 11M06.}

\section{Introduction}
For an even integer $k\geq 2$, denote $H_k$ the set of all normalized Hecke primitive cusp forms of weight $k$ for the full modular group $SL(2,\ \mathbb{Z})$. Throughout this paper we call the function $f(z)$ as a primitive cusp form if it is an eigenfunction of all Hecke operators simultaneously. It is known that $f(z)$ has a Fourier expansion at cusp $\infty$, write it as
\begin{equation*}
    f(z)=\sum_{n=1}^\infty \lambda_f(n)n^{(k-1)/2}e^{2\pi i n z}
\end{equation*}
for $\Im (z)>0$.\\
Then by Deligne \cite{Deligne}, we have, for any prime number $p$, there exists two complex numbers $\alpha_f(p)$ and $\beta_f(p)$, such that 
\begin{equation*}
    \alpha_f(p)\beta_f(p)=|\alpha_f(p)|=|\beta_f(p)|=1
\end{equation*} 
and 
\begin{equation*}
    \lambda_f(p)=\alpha_f(p)+\beta_f(p),
\end{equation*}
and also for holomorphic cusp forms, Deligne \cite{Deligne} proved the Ramanujan- Petersson conjecture
\begin{equation*}
	\mid\lambda_f(n)\mid\leq d(n).
\end{equation*}
Let $f\in H_k$. The Hecke $L$-function attached to $f$ is defined as
\begin{equation*}
    L(s, f)=\sum_{n= 1}^\infty\frac{\lambda_f(n)}{n^s}=\prod_p\left(1-\frac{\alpha_f(p)}{p^s}\right)^{-1}\left(1-\frac{\beta_f(p)}{p^s}\right)^{-1}
\end{equation*}
which converges absolutely for $\Re(s)>1.$\\
The $j^{\textit{th}}$ symmetric power $L$-function attached to $f$ is defined as
\begin{equation*}
        L(s, {\rm{sym}}^jf)=\prod_p\prod_{m=0}^j(1-\alpha_p^{j-2m}p^{-s})^{-1}
\end{equation*}
for $\Re (s)>1$. We may express it as a Dirichlet series: for $\Re (s)>1$, 
\begin{align*}
     L(s,{\rm{sym}}^j f)&=\sum_{n=1}^\infty \frac{\lambda_{{\rm{sym}}^j f}(n)}{n^s}\nonumber\\
     &=\prod_p\bigg(1+\frac{\lambda_{{\rm{sym}}^j f}(p)}{p^s}+\ldots+\frac{\lambda_{{\rm{sym}}^j f}(p^k)}{p^{ks}}+\ldots\bigg).
\end{align*}
It is well known that $\lambda_{{\rm{sym}}^j f}(n)$ is a real multiplicative function.\\
Note that $\displaystyle L(s,{\rm{sym}}^0 f)=\zeta(s)$ (Riemann zeta function) and $L(s,{\rm{sym}}^1 f)=L(s,f)$ (Hecke $L$-function).\\
The twisted $j^{\textit{th}}$ symmetric power $L$-function attached to $f$ by the Dirichlet character $\chi$ is defined as 
\begin{align*}
	L(s,\ {\rm{sym}}^{j}f\otimes\chi)=&\sum_{n=1}^\infty\frac{\lambda_{{\rm{sym}}^{j}f}(n)\chi(n)}{n^s}\\
		=&\prod_p\prod_{m=0}^{j}\left(1-\frac{\alpha_f^{{2j}-2m}(p)\chi(p)}{p^s}\right)^{-1}
\end{align*}
	for $\Re(s)>1$ and $L(s,\ {\rm{sym}}^{j}f\otimes\chi)$ is of degree $j+1$.\\
For any Dirichlet character modulo $q$, the Dirichlet $L$-function is defined as 
\begin{equation*}
    L(s,\chi)=\sum_{n=1}^\infty \frac{\chi(n)}{n^s}=\prod_p \left(1-\frac{\chi(p)}{p^s}\right)^{-1}
\end{equation*}
for $\Re(s)>1$.\\
Throughout the paper we assume $\epsilon>0$ to be small which need not be the same at each occurrence.\\
It has been an interesting and important topic in number theory to study the average behavior of the Fourier coefficients of cusp forms since long time. \\
In the paper \cite{ASAS2022}, A. Sharma and A. Sankaranarayanan proved that 
\begin{equation*}
    \sum_{\substack{n=a^2+b^2+c^2+d^2\leq x\\a,b,c,d\in \mathbb{Z}}}\lambda_{{\rm{sym}}^2  f}^2(n)=c_1x^2+O_f\left(x^{\frac{9}{5}+\epsilon}\right),
\end{equation*}
for sufficiently large $x\geq x_0$.\\
Again in another paper \cite{ASAS2023}, A. Sharma and A. Sankaranarayanan established the following theorem.
\begin{theoremA} 
Let $j\geq 2$ be any fixed integer. For sufficiently large $x$, and $\epsilon>0$ any small constant, we have 
    \begin{equation*}
   \sum_{\substack{n=a_1^2+a_2^2+\hdots+a_6^2\leq x\\\left(a_1,a_2,\hdots, a_6\right)\in \mathbb{Z}^6}}\lambda_{{\rm{sym}}^j  f}^2(n) =c_{f,j} x^3+O\left(x^{3-\frac{6}{3(j+1)^2+1}+\epsilon}\right),
\end{equation*}
where $c_{f,j}$ is an effective constant defined by 
\begin{equation*}
    c_{f,j}=\frac{16}{3}L(3,\chi)\prod_{n=1}^j L(1,{\rm{sym}}^{2n} f) L(3,{\rm{sym}}^{2n} f\otimes \chi)\mathcal{H}_j(3),
\end{equation*}
where $\chi$ is the nonprincipal Dirichlet character modulo $4$.
\end{theoremA}
Later, in the paper \cite{Youjun}, Youjun Wang proved theorem A for $j\geq 2$ with an improved error term.\\
More precisely he established:
\begin{theoremA}
    Let $j\geq 2$ be any fixed integer. For sufficiently large $x$, and  any small constant $\epsilon>0$, we have 
    \begin{equation*}
   \sum_{\substack{n=a_1^2+a_2^2+\hdots+a_6^2\leq x\\\left(a_1,a_2,\hdots, a_6\right)\in \mathbb{Z}^6}}\lambda_{{\rm{sym}}^j  f}^2(n) =c_{f,j} x^3+O\left(x^{3-\frac{6}{3(j+1)^2-1}+\epsilon}\right),
\end{equation*}
where 
\begin{equation*}
    c_{f,j}=\frac{16}{3}L(3,\chi)\prod_{n=1}^j L(1,{\rm{sym}}^{2n} f) L(3,{\rm{sym}}^{2n} f\otimes \chi)\mathcal{H}_j(3),
\end{equation*}
where $\chi$ is the nonprincipal Dirichlet character modulo $4$ and $\mathcal{H}_j(s)$ is a convergent Dirichlet series.\label{TB}
\end{theoremA}
The aim of this article is to improve the error term in Theorem B for $j\geq 3$. Indeed we prove:
 
\begin{theorem}
Let $f\in H_k$ and $j\geq 3$. Then we have 
\begin{equation*}
    \sum_{\substack{n=a_1^2+a_2^2+\hdots+a_6^2\leq x\\\left(a_1,a_2,\hdots, a_6\right)\in \mathbb{Z}^6}}\lambda_{{\rm{sym}}^j  f}^2(n)=\mathcal{C}_{f, j} x^3+O\left(x^{3-\frac{30}{15(j+1)^2-14}+\epsilon}\right),
\end{equation*}
where $\displaystyle \mathcal{C}_{f,j}=\frac{16}{3} L(3, \chi)\prod_{n=1}^j L(1, {\rm{sym}}^{2n}f) L(3, {\rm{sym}}^{2n}f\otimes\chi)H_j(3)$, $\chi$ is the nonprincipal character modulo $4$ and $H_j(s)$ is a convergent Dirichlet series in lemma 2.2.\label{T1.1}
\end{theorem}

\begin{remark}
Note that $3-\frac{30}{15(j+1)^2-14}<3-\frac{6}{3(j+1)^2-1}$ for $j\geq 3$. Hence theorem \ref{T1.1} improves theorem \ref{TB}. One should notice that this improvement is valid for integers $j\geq 3$ and for $j=2$ theorem \ref{TB} is the best known result till now.
\end{remark}

\section{Lemmas}
\begin{lemma}
    Let $f\in H_k$, $j\geq 2$ and $\chi$ be the nonprincipal Dirichlet character modulo $4$. Then we have 
\begin{align*}
    \sum_{\substack{n=a_1^2+a_2^2+\hdots+a_6^2\leq x\\\left(a_1,a_2,\hdots, a_6\right)\in \mathbb{Z}^6}}\lambda_{{\rm{sym}}^j  f}^2(n)&=\sum_{n\leq x}\lambda_{{\rm{sym}}^jf}^2(n) r_6(n)\\
    &=16\sum_{n\leq x}\lambda_{{\rm{sym}}^jf}^2(n) l(n)-4\sum_{n\leq x}\lambda_{{\rm{sym}}^jf}^2(n) \upsilon(n),
\end{align*}
where $\displaystyle l(n)=\sum_{d|n}\chi(d)\frac{n^2}{d^2}$ and $\displaystyle \upsilon(n)=\sum_{d|n}\chi(d)d^2$.
\end{lemma}
\begin{proof}
    See lemma 2.1 of \cite{ASAS2023}.
\end{proof}

\begin{lemma}
    Let $f\in H_k$ and $j\geq 2$. Then for $\Re (s)>3$, we have 
    \begin{align*}
        F_{j}(s)&=\sum_{n=1}^\infty \frac{\lambda_{{\rm{sym}}^jf}^2(n) l(n)}{n^s}\\
        &=G_{j}(s)H_{j}(s),
    \end{align*}
     where
        \begin{equation*}
            G_j(s):=\zeta(s-2)L(s, \chi)\prod_{n=1}^j L(s-2, {\rm{sym}}^{2n}f) L(s, {\rm{sym}}^{2n}f\otimes\chi),
        \end{equation*}
        $\chi$ is the nonprincipal character modulo 4 and $H_j(s)$ is some Dirichlet series which converges absolutely and uniformly in the half plane $\Re (s)>\frac{5}{2}$, and $H_j(s)\neq 0$ on $\Re (s)=3$.
\end{lemma}
\begin{proof}
    See lemma 2.2 of \cite{ASAS2023}.
\end{proof}

\begin{lemma}
     Let $f\in H_k$ and $j\geq 2$. Then for $\Re (s)>3$, we have 
    \begin{align*}
        \widetilde{F}_{j}(s)&=\sum_{n=1}^\infty \frac{\lambda_{{\rm{sym}}^jf}^2(n) \upsilon(n)}{n^s}\\
        &= \widetilde{G}_{j}(s) \widetilde{H}_{j}(s),
    \end{align*}
     where
        \begin{equation*}
             \widetilde{G}_j(s):=\zeta(s)L(s-2, \chi)\prod_{n=1}^j L(s, {\rm{sym}}^{2n}f) L(s-2, {\rm{sym}}^{2n}f\otimes\chi),
        \end{equation*}
        $\chi$ is the nonprincipal character modulo 4 and $ \widetilde{H}_j(s)$ is some Dirichlet series which converges absolutely and uniformly in the half plane $\Re (s)>\frac{5}{2}$, and $H_j(s)\neq 0$ on $\Re (s)=3$.
\end{lemma}
\begin{proof}
    See lemma 2.3 of \cite{ASAS2023}.
\end{proof}

\begin{lemma}
Suppose that $\mathfrak{L}(s)$ is a general $L$-function of degree $m$. Then, for any $\epsilon>0,$ we have 
\begin{equation}
    \int_T^{2T}|\mathfrak{L}(\sigma+it)|^2  \ dt\ll T^{\max\{m(1-\sigma),1\}+\epsilon}\label{E1}
\end{equation}
uniformly for $\frac{1}{2}\leq \sigma \leq1$ and $| t|\geq 1$; and 
\begin{equation}
    \mathfrak{L}(\sigma+it)\ll(|t|+1)^{\frac{m}{2}(1-\sigma)+\epsilon}\label{E2}
\end{equation}
uniformly for $\frac{1}{2}\leq \sigma \leq1+\epsilon$ and $| t|\geq 1$.
\end{lemma}
\begin{proof}
    The result \eqref{E1} is due to Perelli \cite{Perelli} and \eqref{E2} follows from Maximum modulus principle.
\end{proof}

 For some L-functions with small degrees, we invoke either individual or average subconvexity bounds.

\begin{lemma}
For $\frac{5}{8}\leq \sigma\leq\frac{35}{54}$, we have \begin{equation}
    \int_1^T|\zeta\left(\sigma+it\right)|^{m(\sigma)}\ dt\ll T^{1+\epsilon},\ m(\sigma)\geq \frac{10}{5-6\sigma},\label{E3}
\end{equation} uniformly for $T\geq1$.
\end{lemma}
\begin{proof}
See theorem $8.4$  of \cite{Ivic}.
\end{proof}
\begin{lemma}
    Let $\chi$ be a primitive character modulo $q$. Then for $q\ll T^2$, we have
\begin{equation}
	L(\sigma+it,\ \chi)\ll (q(1+| t|))^{\max\{\frac{1}{3}(1-\sigma),\ 0\}+\epsilon}\label{E4}
\end{equation}
		holds uniformly for $\frac{1}{2}\leq\sigma\leq 2$ and $| t|\geq 1$;
\end{lemma}

\begin{proof}
    See \cite{Heath-Brown}.
\end{proof}
\begin{lemma}
For $f\in H_k$ and $\epsilon>0$, we have 
\begin{equation}
    L(\sigma+it,\ {\rm{sym}}^2 f)\ll_{f,\epsilon}(|t|+1)^{\max\{\frac{6}{5}(1-\sigma),0\}+\epsilon}\label{E5}
\end{equation}
uniformly for $\frac{1}{2}\leq \sigma \leq2$ and $| t|\geq 1$.
\end{lemma}
\begin{proof}
See \cite{Lin}.
\end{proof}

\begin{lemma}[KR+AS]
For $\frac{1}{2}\leq \sigma\leq 2$, T-sufficiently large, there exist a $T^*\in[T,T+T^\frac{1}{3}]$ such that the bound 
\begin{equation*}
    \log\zeta(\sigma+iT)\ll (\log\log T^*)^2\ll(\log\log T)^2
\end{equation*}
holds uniformly and we have 
\begin{equation}
    \zeta(\sigma+iT)\ll \exp((\log\log T)^2)\ll_\epsilon T^\epsilon\label{E6}
\end{equation}
on the horizontal line with $T=T^*$ and $\frac{1}{2}\leq \sigma\leq 2.$
\end{lemma}
\begin{proof}
See lemma $1$ of \cite{KRAS}.
\end{proof}

\begin{lemma}
    Let $\lambda>0,\mu>0$ and $\alpha<\sigma< \beta$. Then we have 
\begin{align}
    J(\sigma, p\lambda+q\mu)=O\{J^p(\alpha, \lambda)J^q(\beta, \mu)\},\label{E7}
\end{align}
where $\displaystyle J(\sigma, \lambda)=\left\{\int_0^T| f(\sigma+it)|^\frac{1}{\lambda}dt \right\}^\lambda$, $p=\frac{\beta-\sigma}{\beta-\alpha}$ and $q=\frac{\sigma-\alpha}{\beta-\alpha}$.
\end{lemma}
\begin{proof}
    See pp. 236 of \cite{Titchmarsh}
\end{proof}

\begin{lemma}
    Let $f\in H_k$ and $j\geq 3$. Then we have 
    \begin{equation}
        \left\{\int_0^T\prod_{n=1}^{j-1}\left| L\left(\frac{5}{8}+\epsilon+it, {\rm{sym}}^{2n}f\right)\right|^\frac{40}{15+24\epsilon}dt\right\}^\frac{15+24\epsilon}{40}\ll_{f,\epsilon} T^{\frac{3j^2}{16}-\frac{3}{10}+\epsilon}.\label{E8}
    \end{equation}
\end{lemma}
\begin{proof}
    We prove this lemma by using  the lemma 2.9. For that, we choose the parameters suitably as follows.\\
    We choose $\alpha=\frac{1}{2}+\epsilon,\ \sigma=\frac{5}{8}+\epsilon,\ \beta=1+\epsilon $ and $\lambda=\frac{1}{2}$. Then we get $p=\frac{\beta-\sigma}{\beta-\alpha}=\frac{3}{4}$ and $q=1-p=\frac{1}{4}$.\\
    Now we choose $\mu=\frac{24\epsilon}{10}>0$, so that we get $p\lambda+q\mu=\frac{15+24\epsilon}{40}$.\\
    Therefore, we have 
    \begin{align*}
        &   \left\{\int_0^T\prod_{n=1}^{j-1}\left| L\left(\frac{5}{8}+2\epsilon+it, {\rm{sym}}^{2n}f\right)\right|^\frac{40}{15+24\epsilon}dt\right\}^\frac{15+24\epsilon}{40}\\
        &\ll \left\{\int_0^T\prod_{n=1}^{j-1}\left| L\left(\frac{1}{2}+\epsilon+it, {\rm{sym}}^{2n}f\right)\right|^\frac{1}{\lambda} dt\right\}^{p\lambda} \left\{\int_0^T\prod_{n=1}^{j-1}\left| L\left(1+\epsilon+it, {\rm{sym}}^{2n}f\right)\right|^\frac{1}{\mu}dt\right\}^{q\mu}\\
        &\ll_{f,\epsilon} \left\{\max_{0\leq t\leq T}\left|L\left(\frac{1}{2}+\epsilon+it, {\rm{sym}}^2f\right)\right|^2\int_0^T\prod_{n=2}^{j-1}\left| L\left(\frac{1}{2}+\epsilon+it, {\rm{sym}}^{2n}f\right)\right|^2 dt\right\}^\frac{3}{8}\times T^{\frac{24\epsilon}{40}}\\
        &\ll_{f,\epsilon} T^{\left(\frac{12}{5}+(j^2-4)\right)(\frac{1}{2}-\epsilon)\frac{3}{8}+\frac{24\epsilon}{40}}\\
        &\ll_{f,\epsilon} T^{\frac{3j^2}{16}-\frac{3}{10}+\epsilon},
    \end{align*} 
    which followed by the lemmas 2.4 and 2.7, and the fact that the $L$-function $\displaystyle\prod_{n=2}^{j-1} L\left(s, {\rm{sym}}^{2n}f\right)$ is of degree $j^2-4$, and $\displaystyle\prod_{n=1}^{j-1}\left| L\left(1+\epsilon+it, {\rm{sym}}^{2n}f\right)\right|\ll_{f, \epsilon} 1$ for every integer $j\geq 3$.
\end{proof}

\section{Proof of Theorem 1.1}
Let $j\geq 3$. By lemma 2.1, we have 
\begin{align}
    \sum_{\substack{n=a_1^2+a_2^2+\hdots+a_6^2\leq x\\\left(a_1,a_2,\hdots, a_6\right)\in \mathbb{Z}^6}}\lambda_{{\rm{sym}}^j  f}^2(n)&=16\sum_{n\leq x}\lambda_{{\rm{sym}}^jf}^2(n) l(n)-4\sum_{n\leq x}\lambda_{{\rm{sym}}^jf}^2(n) \upsilon(n)\nonumber\\
   &:=\textit{$\sum$}_1+\textit{$\sum$}_2,\label{E9}
\end{align}
where $\displaystyle l(n)=\sum_{d|n}\chi(d)\frac{n^2}{d^2}$ and $\displaystyle \upsilon(n)=\sum_{d|n}\chi(d)d^2$.\\
For $\displaystyle\textit{$\sum$}_1$ by applying Perron's formula to $F_j(s)$, by lemma 2.2, we get
\begin{align*}
    \textit{$\sum$}_1&=\frac{16}{2\pi i} \int_{3+\epsilon-iT}^{3+\epsilon+iT} F_j(s)\frac{x^s}{s} ds+O\left(\frac{x^{3+\epsilon}}{T}\right),
\end{align*}
where $10\leq T\leq x$ is a parameter to be chosen later suitably, we make the special choice of $T=T^*$ of lemma 2.8 satisfying \eqref{E6}. \\
We move the line of integration to $\Re(s)=2+\frac{5}{8}+\epsilon$, probably $\epsilon<\frac{5}{216}$, so that the lemma 2.5 can be used. In the rectangle formed by the line segments joining the points $3+\epsilon-iT,\ 3+\epsilon+iT,\ 2+\frac{5}{8}+\epsilon+iT,\ 2+\frac{5}{8}+\epsilon-iT,\ 3+\epsilon-iT$, we note that $F_j(s)$ is a meromorphic function having a simple pole at $s=3$. \\
Thus by Cauchy's residue theorem, we have
\begin{align*}
    \textit{$\sum$}_1&=\mathcal{C}_{f, j}x^3 +\frac{16}{2\pi i}\left\{\int_{2+\frac{5}{8}+\epsilon-iT}^{2+\frac{5}{8}+\epsilon+iT}+\int_{2+\frac{5}{8}+\epsilon-iT}^{3+\epsilon-iT}+\int_{2+\frac{5}{8}+\epsilon+iT}^{3+\epsilon+iT}\right\}F_j(s)\frac{x^s}{s} ds+O\left(\frac{x^{3+\epsilon}}{T}\right)\\
    &:=\mathcal{C}_{f, j}x^3 +I_1+I_2+I_3+O\left(\frac{x^{3+\epsilon}}{T}\right),
\end{align*}
where $\displaystyle \mathcal{C}_{f,j}x^3=16{\it Res}_{s=3}F_j(s)\frac{x^s}{s}=\frac{16}{3}x^3 L(3, \chi)\prod_{n=1}^j L(1, {\rm{sym}}^{2n}f) L(3, {\rm{sym}}^{2n}f\otimes\chi)H_j(3)$.\\
By using the lemmas 2.2, 2.4, 2.5, 2.10 and the H{\"o}lders inequality, we derive the vertical line contribution $I_1$ as follows:
\begin{align*}
    I_1&\ll \int_{2+\frac{5}{8}+\epsilon-iT}^{2+\frac{5}{8}+\epsilon+iT} \left|\zeta(s-2)\prod_{n=1}^j L(s-2, {\rm{sym}}^{2n}f)\frac{x^s}{s}ds\right|\\
    &\ll \int_{-T}^{T} \left|\zeta\left(\frac{5}{8}+\epsilon+it\right) \prod_{n=1}^j L\left(\frac{5}{8}+\epsilon+it, {\rm{sym}}^{2n}f\right) \right|\frac{x^{2+\frac{5}{8}+\epsilon}}{|2+\frac{5}{8}+\epsilon+it|} dt\\
    &\ll x^{2+\frac{5}{8}+\epsilon} +x^{2+\frac{5}{8}+\epsilon} \int_{10}^T \left|\zeta\left(\frac{5}{8}+\epsilon+it\right) \prod_{n=1}^j L\left(\frac{5}{8}+\epsilon+it, {\rm{sym}}^{2n}f\right) \right| t^{-1}dt\\
    &\ll x^{2+\frac{5}{8}+\epsilon}+x^{2+\frac{5}{8}+\epsilon} \sup_{10\leq T_1\leq T} \left\{\int_0^T\prod_{n=1}^{j-1}\left| L\left(\frac{5}{8}+\epsilon+it, {\rm{sym}}^{2n}f\right)\right|^\frac{40}{15+24\epsilon}dt\right\}^\frac{15+24\epsilon}{40}\times\\
    &\qquad \left\{\int_{T_1}^{2T_1}\left|L\left(\frac{5}{8}+\epsilon+it,{\rm{sym}}^{2j}f \right)\right|^2\right\}^\frac{1}{2} \left\{\int_{T_1}^{2T_1}\left|\zeta\left(\frac{5}{8}+\epsilon+it\right)\right|^\frac{40}{5-24\epsilon}\right\}^\frac{5-24\epsilon}{40} T_1^{-1}\\
    &\ll_{f,\epsilon} x^{2+\frac{5}{8}+\epsilon}+x^{2+\frac{5}{8}+\epsilon} T^{\frac{5-24\epsilon}{40}+\frac{2j+1}{2}(\frac{3}{8}-\epsilon)+\frac{3j^2}{16}-\frac{3}{10}-1+\epsilon}\\
    &\ll_{f,\epsilon} x^{2+\frac{5}{8}+\epsilon} T^{\frac{3(j+1)^2}{16}-\frac{7}{40}-1+\epsilon}.
\end{align*}
Now by using the lemmas 2.4, 2.7 and 2.8 we derive the horizontal lines contribution $I_2+I_3$, as follows:
\begin{align*}
    I_2+I_3&\ll \int_{2+\frac{5}{8}+\epsilon+iT}^{3+\epsilon+iT} \left|\zeta(s-2)\prod_{n=1}^j L(s-2, {\rm{sym}}^{2n}f)\frac{x^s}{s}ds\right|\\
    &\ll \int_{\frac{5}{8}+\epsilon}^{1+\epsilon} \left|\zeta(\sigma+iT)\prod_{n=1}^j L(\sigma+iT, {\rm{sym}}^{2n}f)\right|x^{2+\sigma} T^{-1} d\sigma\\
    &\ll \int_{\frac{5}{8}+\epsilon}^{1+\epsilon} T^{\epsilon+\frac{6}{5}(1-\sigma)+\frac{(j+1)^2-4}{2}(1-\sigma)-1} x^{2+\sigma} d\sigma\\
    &\ll x^2\max_{\frac{5}{8}+\epsilon\leq\sigma\leq 1+\epsilon} \left(\frac{x}{T^{\frac{(j+1)^2}{2}-\frac{4}{5}}}\right)^\sigma T^{\frac{(j+1)^2}{2}-\frac{4}{5}-1+\epsilon}\\
    &\ll \frac{x^{3+\epsilon}}{T}+x^{2+\frac{5}{8}+\epsilon} T^{\frac{3(j+1)^2}{16}-\frac{3}{10}-1+\epsilon},
\end{align*}
which followed by the fact that the $L$-function $\displaystyle\prod_{n=2}^j L(s, {\rm{sym}}^{2n}f)$ is of degree $(j+1)^2-4$.\\
Therefore, we have
\begin{equation*}
     \textit{$\sum$}_1=\mathcal{C}_{f, j}x^3 +O\left(x^{2+\frac{5}{8}+\epsilon} T^{\frac{3(j+1)^2}{16}-\frac{7}{40}-1+\epsilon}\right)+O\left(\frac{x^{3+\epsilon}}{T}\right).
\end{equation*} 
Thus, by choosing $\displaystyle T=x^{\frac{30}{15(j+1)^2-14}}$, we obtain
\begin{equation}
     \textit{$\sum$}_1=\mathcal{C}_{f, j}x^3 +O\left(x^{3-\frac{30}{15(j+1)^2-14}}
\right),\label{E10}
\end{equation} 
where $\displaystyle \mathcal{C}_{f,j}x^3=\frac{16}{3}x^3 L(3, \chi)\prod_{n=1}^j L(1, {\rm{sym}}^{2n}f) L(3, {\rm{sym}}^{2n}f\otimes\chi)H_j(3)$.\\
Now for $\displaystyle\textit{$\sum$}_2$ by applying the Perron's formula to $\widetilde{F}_{j}(s)$, by lemma 2.3, we get
\begin{align*}
    \textit{$\sum$}_2&=\frac{-4}{2\pi i} \int_{3+\epsilon-iT}^{3+\epsilon+iT} \widetilde{F}_{j}(s)\frac{x^s}{s} ds+O\left(\frac{x^{3+\epsilon}}{T}\right),
\end{align*}
where $10\leq T\leq x$ is a parameter to be chosen later suitably.\\
We move the line of integration to $\Re(s)=2+\frac{2}{3}$. In the rectangle formed by the line segments joining the points $3+\epsilon-iT,\ 3+\epsilon+iT,\ 2+\frac{2}{3}+iT,\ 2+\frac{2}{3}-iT,\ 3+\epsilon-iT$, we note that $\widetilde{F}_{j}(s)$ does not have any singularities.
Thus by Cauchy's theorem, we have
\begin{align*}
    \textit{$\sum$}_2&\ll \left\{\int_{2+\frac{2}{3}-iT}^{2+\frac{2}{3}+iT}+\int_{2+\frac{2}{3}-iT}^{3+\epsilon-iT}+\int_{2+\frac{2}{3}+iT}^{3+\epsilon+iT}\right\}\widetilde{F}_{j}(s)\frac{x^s}{s} ds+O\left(\frac{x^{3+\epsilon}}{T}\right)\\
    &:= J_1+J_2+J_3+O\left(\frac{x^{3+\epsilon}}{T}\right).
\end{align*}
By using the lemmas 2.4, 2.6, 2.7 and the Cauchy's inequality, we derive the vertical line contribution $J_1$ as follows: 
\begin{align*}
    J_1&\ll \int_{2+\frac{2}{3}-iT}^{2+\frac{2}{3}+iT} \left|L(s-2,\  \chi)\prod_{n=1}^j L(s-2, {\rm{sym}}^{2n}f\otimes \chi)\frac{x^s}{s}ds\right|\\
    &\ll \int_{-T}^{T} \left|L\left(\frac{2}{3}+it,\  \chi\right) \prod_{n=1}^j L\left(\frac{2}{3}+it, {\rm{sym}}^{2n}f\otimes \chi\right) \right|\frac{x^{2+\frac{2}{3}}}{|2+\frac{2}{3}+it|} dt\\
    &\ll x^{2+\frac{2}{3}} +x^{2+\frac{2}{3}} \int_{10}^T \left|L\left(\frac{2}{3}+it,\  \chi\right) \prod_{n=1}^j L\left(\frac{2}{3}+it, {\rm{sym}}^{2n}f\otimes \chi\right) \right| t^{-1}dt\\
    &\ll x^{2+\frac{2}{3}}+x^{2+\frac{2}{3}} \sup_{10\leq T_1\leq T} T_1^{-1}\left\{\max_{T_1\leq t\leq 2T_1 }\left|L\left(\frac{2}{3}+it,\ \chi\right)L\left(\frac{2}{3}+it, {\rm{sym}}^{2}f\otimes \chi\right)\right| \right\} \times\\
    & \qquad\left\{\int_{T_1}^{2T_1}\left| L\left(\frac{2}{3}+it, {\rm{sym}}^4f\otimes \chi\right)\right|^2dt\right\}^\frac{1}{2} \left\{\int_{T_1}^{2T_1}\prod_{n=3}^{j}\left| L\left(\frac{2}{3}+it, {\rm{sym}}^{2n}f\otimes \chi\right)\right|^2dt\right\}^\frac{1}{2}\\
    &\ll x^{2+\frac{2}{3}}+x^{2+\frac{2}{3}} T^{\frac{1}{9}+\frac{2}{5}+\frac{5}{6}+\frac{(j+1)^2-9}{2}\frac{1}{3}-1+\epsilon}\\
    &\ll x^{2+\frac{2}{3}} T^{\frac{(j+1)^2}{6}-\frac{7}{45}-1+\epsilon},
\end{align*}
which followed by the facts that the $L$-function $\displaystyle\prod_{n=3}^j L(s, {\rm{sym}}^{2n}f)$ is of degree $(j+1)^2-9$ and  $\chi$ is the nonprincipal Dirichlet character modulo $4$.\\
Now by using the lemmas 2.4, 2.6 and 2.7 we derive the horizontal lines contribution $J_2+J_3$, as follows:
\begin{align*}
    J_2+J_3&\ll \int_{2+\frac{2}{3}+iT}^{3+\epsilon+iT} \left|L(s-2,\  \chi)\prod_{n=1}^j L(s-2, {\rm{sym}}^{2n}f\otimes \chi)\frac{x^s}{s}ds\right|\\
    &\ll \int_{\frac{2}{3}}^{1+\epsilon} \left|L(\sigma+iT,\  \chi)\prod_{n=1}^j L(\sigma+iT, {\rm{sym}}^{2n}f\otimes \chi)\right|x^{2+\sigma} T^{-1} d\sigma\\
    &\ll \int_{\frac{2}{3}}^{1+\epsilon} T^{\frac{1-\sigma}{3}+\frac{6}{5}(1-\sigma)+\frac{(j+1)^2-4}{2}(1-\sigma)-1} x^{2+\sigma} d\sigma\\
    &\ll x^2\max_{\frac{2}{3}\leq\sigma\leq 1+\epsilon} \left(\frac{x}{T^{\frac{(j+1)^2}{2}-\frac{7}{15}}}\right)^\sigma T^{\frac{(j+1)^2}{2}-\frac{7}{15}-1+\epsilon}\\
    &\ll \frac{x^{3+\epsilon}}{T}+x^{2+\frac{2}{3}} T^{\frac{(j+1)^2}{6}-\frac{7}{45}-1+\epsilon},
\end{align*}
which followed by the facts that the $L$-function $\displaystyle\prod_{n=2}^j L(s, {\rm{sym}}^{2n}f)$ is of degree $(j+1)^2-4$ and  $\chi$ is the nonprincipal Dirichlet character modulo $4$.\\
Therefore, we have
\begin{equation*}
    \textit{$\sum$}_2\ll \frac{x^{3+\epsilon}}{T}+x^{2+\frac{2}{3}} T^{\frac{(j+1)^2}{6}-\frac{7}{45}-1+\epsilon}.
\end{equation*}
Thus by choosing $T=x^{\frac{30}{15(j+1)^2-14}}$, we obtain 
\begin{equation}
    \textit{$\sum$}_2\ll x^{3-\frac{30}{15(j+1)^2-14}}.\label{E11}
\end{equation}
Combining \eqref{E9}, \eqref{E10} and \eqref{E11}, we get 
\begin{equation*}
    \sum_{\substack{n=a_1^2+a_2^2+\hdots+a_6^2\leq x\\\left(a_1,a_2,\hdots, a_6\right)\in \mathbb{Z}^6}}\lambda_{{\rm{sym}}^j  f}^2(n)=\mathcal{C}_{f, j}x^3 +O\left(x^{3-\frac{30}{15(j+1)^2-14}}\right),
\end{equation*}
where $\displaystyle \mathcal{C}_{f,j}x^3=\frac{16}{3}x^3 L(3, \chi)\prod_{n=1}^j L(1, {\rm{sym}}^{2n}f) L(3, {\rm{sym}}^{2n}f\otimes\chi)H_j(3)$.\\
This proves theorem \ref{T1.1}. \qed

\textbf{Acknowledgements:} The first author wishes to express his gratitude to the Funding Agency ``Ministry of Education (MoE), Govt. of India" for the fellowship PMRF, Id: 3701831 for its financial support.


\begin{thebibliography}{}
\bibitem[A]{Deligne} P. Deligne, \emph{La conjecture de Weil}, I, II, Publ. Math. IHES., \textbf{43} (1974), 273-308; ibid \textbf{52} (1981), 313-428.
\bibitem[B]{Heath-Brown} D. R. Heath-Brown, \emph{Hybrid bounds for $L$-functions}, Invert. Math., {\bf 47}  (1978) 149-170.
\bibitem[C]{Ivic} A. Ivi{\'c}, \emph{Exponent pairs and the zeta function of Riemann}, Stud. Sci. Math. Hungar., \textbf{15} (1980), 157-181.
\bibitem[D]{Lin} Y. X. Lin, R. Nunes, and Z. Qi, \emph{Strong subconvexity for self-dual $GL(3)$ $L$-functions}, Int. Math. Res. Not., {\bf 13} (2023), 11453-11470.
\bibitem[E]{Perelli} A. Perelli, \emph{General $L$-functions}, Ann. Mat. Pura Appl., {\bf 130} (1982) 287-306.
\bibitem[F]{KRAS} K. Ramachandra and A. Sankaranarayanan, \emph{Notes on the Riemann zeta-function}, Journal of Indian Math. soc., \textbf{57} (1991), 67-77.
\bibitem[G]{ASAS2022} A. Sharma and A. Sankaranarayanan, \emph{Discrete mean square of the coefficients of symmetric square $L$-functions on certain sequence of positive numbers}, Res. Number Theory, {\bf 8} (2022), 19.
\bibitem[H]{ASAS2023} A. Sharma and A. Sankaranarayanan, \emph{On the average behavior of the Fourier coefficients of $j$th symmetric power L-function over certain sequences of positive integers}, Czech, Math. J., {\bf 73} (2023), 885-901.
\bibitem[I]{Titchmarsh} E. C. Titchmarsh and D. R. Heath-Brown, \emph{The theory of the Riemann zeta-function}, second edition, Clarendon press, Oxford (1986).
\bibitem[J]{Youjun} Youjun Wang, \emph{A note on average behaviour of the Fourier coefficients of $j^\textit{th}$ symmetric power $L$-function over certain sparse sequence of positive integers}, Czech. Math. J., {\bf 74} (2024), 623–636.



\end{thebibliography}
\end{document}